\documentclass[12pt,titlepage]{article}

\usepackage{graphicx,color}
\usepackage{amsmath,amssymb,amsthm}
\newcommand{\proofka}{{\noindent \it Proof. }}
\newtheorem{thm}{Theorem}
\newtheorem{lemma}[thm]{Lemma}
\newtheorem{cor}[thm]{Corollary}
\newtheorem{problem}{Problem}

\newtheorem{prop}[thm]{Proposition}

\newtheorem{defi}{Definition}

\newtheorem{remark}{Remark}

\def\beq{\begin{equation}}\def\eeq{\end{equation}}
\def\beqn{\begin{eqnarray}}\def\eeqn{\end{eqnarray}}

\def\eps{\varepsilon}
\def\qed{\ifhmode\unskip\nobreak\fi\quad\ifmmode\Box\else$\Box$\fi}
\newcommand{\M}{{\cal M}}
\newcommand{\A}{{\cal A}}
\newcommand{\B}{{\cal B}}
\newcommand{\G}{{\cal G}}
\newcommand{\F}{{\cal F}}
\newcommand{\LL}{{\cal L}}
\newcommand{\C}{{\cal C}}

\newcommand\mbf[1]{\mbox{\boldmath$#1$}}
\newcommand\msbf[1]{\mbox{\boldmath\scriptsize$#1$}}

\title{Structured Codes of Graphs}
\author{\hfil Noga Alon\thanks{Department of Mathematics, Princeton University, Princeton, NJ 08544, USA and Schools of Mathematics and
Computer Science, Tel Aviv University, Tel Aviv 69978,
Israel.
Research supported in part by NSF grant DMS-1855464
and BSF grant 2018267; {\tt nalon@math.princeton.edu}}
\and Anna Gujgiczer\thanks
{Department of Computer Science and Information Theory,
Faculty of Electrical Engineering and Informatics,
Budapest University of Technology and Economics and
MTA-BME Lend\"ulet Arithmetic Combinatorics Research Group, Budapest, Hungary. Research partially supported by the National Research, Development and Innovation Office (NKFIH) grant K--120706 of NKFIH Hungary; {\tt gujgicza@cs.bme.hu}
}
\and J\'anos K\"orner\thanks{
Sapienza University of Rome,
Italy; {\tt korner@di.uniroma1.it} }
\and Aleksa Milojevi\'c\thanks{Department of Mathematics, Princeton University, Princeton, NJ 08544, USA; {\tt aleksam@princeton.edu}}
\and G\'abor Simonyi\thanks{Alfr\'ed R\'enyi Institute of Mathematics, Budapest, Hungary and
Department of Computer Science and Information Theory,
Faculty of Electrical Engineering and Informatics,
Budapest University of Technology and Economics. Research partially supported by the National Research, Development and Innovation Office (NKFIH) grants K--120706, K--132696 and SNN-135643 of NKFIH Hungary; {\tt simonyi@renyi.hu}
}
}

\date{}

\begin{document}

\maketitle

\begin{abstract}
We investigate the maximum size of graph families on a common vertex set of cardinality $n$ such that the symmetric difference of the edge sets of any two members of the family satisfies some prescribed condition. We solve the problem completely for infinitely many values of $n$ when the prescribed condition is connectivity or $2$-connectivity, Hamiltonicity or the containment of a spanning star.
We also investigate local conditions that can be certified by looking at only a subset of the vertex set. In these cases a capacity-type asymptotic invariant is defined and when the condition is to contain a certain subgraph this invariant is shown to be a simple function of the chromatic number of this required subgraph. This is proven using classical results from extremal graph theory. Several variants are considered and the paper ends with a collection of open problems.


\end{abstract}

\section{Introduction}

Celebrated problems of extremal combinatorics may get an exciting new flavour when the presence of some special structure is imposed in the condition. A prominent example is the famous Simonovits-S\'os conjecture \cite{SimSos} proven by Ellis,
Filmus and Friedgut \cite{EFF},
which determines the maximum possible cardinality of a family
of graphs on $n$ labeled vertices in which
the intersection of any two members contains a triangle. (The result of \cite{EFF} shows, along with several far reaching generalizations, that the best is to take all graphs containing a given triangle, just as it was conjectured in \cite{SimSos}. This is clearly reminiscent of the Erd\H{o}s-Ko-Rado theorem \cite{EKR}.) As another example we can also mention the Ramsey type problem investigated in \cite{trif} that was also initiated by a question of
S\'os and can be considered as a graph version of the first unsolved case of the so-called perfect hashing problem. (For details we refer to \cite{trif}).

In this paper we study several problems we arrive to if the basic code distance problem (how many binary sequences of a given length can be given at most if any two differ in at least a given number of coordinates) is
modified so that we do not prescribe the minimum distance of any two
codewords but require that they  differ in some specific structure. In particular, just as in the Simonovits-S\'os problem we seek the largest family of (not necessarily induced) subgraphs of a complete graph such that the symmetric difference of the edge sets of any two graphs in the family has some required property. We will consider properties like connectedness, Hamiltonicity, containment of a triangle and some more. Formally all these can be described by saying that the graph defined by the symmetric difference of the edge sets of any two of our graphs belongs to a prescribed family of graphs (namely those that are connected, contain a Hamiltonian cycle, or contain a triangle, etc.)

Let ${\cal F}$ be a fixed class of graphs. A graph family ${\cal G}$ on $n$ labeled vertices is called ${\cal F}$-good if for any pair $G,G'\in {\cal G}$ the graph $G\oplus G'$ defined by $$V(G\oplus G')=V(G)=V(G')=[n],$$ where $[n]=\{1,\dots,n\}$ and $$E(G\oplus G')=\{e: e\in (E(G)\setminus E(G'))\cup (E(G')\setminus E(G))\}$$ belongs to $\F$.

Let $M_{\cal F}(n)$ denote the maximum possible size of an ${\cal F}$-good family on $n$ vertices. We are interested in the value of $M_{\cal F}(n)$ for various classes ${\cal F}$. We will give exact answers or both lower and upper bounds in several cases.

We mention that codes where the codewords are described by graphs already appear in the literature. In \cite{Tonchev}, for example, Tonchev looked at the usual code distance problem restricted to codes whose codewords are characteristic vectors of edge sets of graphs. Gray codes on graphs are also considered, see \cite{Mutze}, where the graphs representing the codewords should have some similarity properties if they are consecutive in a certain listing. Problems analogous to
the present ones though restricted to special graph classes were also considered in \cite{KMS} and \cite{CFK}. A very interesting result along these lines is the one in \cite{Kose}.

\medskip
\par\noindent
The paper is organised as follows. In Section~\ref{gbound} we give a general upper bound that will turn out to be sharp in several of the cases we consider.
In Section~\ref{global} we consider classes ${\cal F}$ defined by some global criterion as connectivity or $2$-connectivity, Hamiltonicity or containing a full star, that is, a vertex of degree $n-1$. We determine $M_{\cal F}(n)$ for infinitely many values of $n$ and for all $n$ in the first and the last case. In most of the cases when we give sharp bounds it is via also solving the problem we call dual: we give the largest possible size of a graph family for which the symmetric difference of no two of its members satisfies the original requirement. The case of the full star is an exception in this sense, nevertheless we also solve the dual problem in that case for all even $n$ by using a celebrated lemma of Shearer.
In Section~\ref{local} we consider classes ${\cal F}$ defined by local conditions. This means that for certifying the condition it is enough to see just a special part of the graph pair at hand.  A capacity-type asymptotic invariant is natural to define in these cases. It turns out that when the requirement is that the pairwise symmetric differences contain a certain subgraph then this asymptotic invariant depends only on the chromatic number of the graph to be contained. We also discuss the case when the above containment is required in an induced manner, and obtain similar results in this case. The final section contains a collection of open problems.

\section{A general upper bound}\label{gbound}

To bound $M_{\cal F}(n)$ for various graph classes ${\cal F}$ it will often be useful to also consider the related problem of constructing large graph families in which no pair satisfies the condition prescribed by ${\cal F}$.

\begin{defi}\label{dual}
For a class of graphs ${\cal F}$ let $D_{\F}(n)$ denote the maximum possible size of a graph family on $n$ labeled vertices (that is, each member of the family has $[n]=\{1,\dots,n\}$ as vertex set), the symmetric difference of no two members of which belongs to $\F$. Determining $D_{{\F}}(n)$ will be referred to as the dual problem of determining $M_{\F}(n)$.
\end{defi}

\medskip
\par\noindent
Note that denoting by $\overline{\F}$ the class containing exactly those graphs that do not belong to $\F$
we actually have $$D_{\F}(n)=M_{\overline{\F}}(n),$$ that is the requirement of having no symmetric difference in $\F$ is clearly the same as saying that all symmetric differences belong to the complementary family $\overline{\F}$. Nevertheless, we will use the $D_{\F}(n)$ notation to emphasize the dual nature of the problem in those cases.

\begin{lemma}\label{lem:ub}
For any graph class $\F$ we have $$M_{\F}(n)\cdot D_{\F}(n)\le 2^{n\choose 2}.$$
\end{lemma}

\proofka
Let us define a graph $H_{\F}$ whose vertices are all the possible (simple) graphs on the vertex set $[n]$. Connect two such vertices if and only if the corresponding pair of graphs have their symmetric difference belonging to $\F$. Then by definition we have $$M_{\F}(n)=\omega(H_{\F}) \ \  {\rm and} \ \ D_{{\F}}(n)=\alpha(H_{\F}),$$
where $\omega(H)$ and $\alpha(H)$ denote the clique number and the independence number of the graph $H$, respectively.
Observe that $H_{\F}$ is vertex-transitive, (in fact it is a Cayley graph
of the group $Z_2^{{n \choose 2}}$). Indeed, if $G_1$ and $G_2$ are two graphs forming vertices of $H_{\F}$ then taking the symmetric difference of all $n$-vertex graphs forming vertices of $H_{\F}$ with the graph $G_1\oplus G_2$ is an automorphism of $H_{\F}$ that maps $G_1$ to $G_2$.
Since a vertex-transitive graph $H$ always satisfies $\alpha(H)\omega(H)\le |V(H)|$ (this can be seen by using that the fractional chromatic number $\chi_f(H)$ always satisfies $\omega(H)\le\chi_f(H)$, while if $H$ is a vertex-transitive graph we also have $\chi_f(H)=\frac{|V(H)|}{\alpha(H)}$, cf. \cite{SchU}), the statement follows.
\qed

\medskip
\par\noindent
The above lemma makes it possible to bound $M_{\F}(n)$ from above by bounding $D_{{\F}}(n)$ from below. In particular, whenever we construct two families of graphs $\A$ and $\B$ on $[n]$ such that $A,A'\in \A$ implies $A\oplus A'\in \F$ and $B,B'\in\B$ implies $B\oplus B'\notin\F$, while $|\A||\B|=2^{n\choose 2}$, then we know that $|\A|$ and $|\B|$ realize the optimal values $M_{\F}(n)$ and $D_{{\F}}(n)$ for such families. Below we will see several cases when this simple technique can indeed be used to obtain these optimal values. An exception to this phenomenon is also presented by Theorems~\ref{thm:fullstar} and \ref{thm:nostar}.

\medskip
\par\noindent
\begin{remark}\label{rem:altproof} {\rm It is worth noting that Lemma~\ref{lem:ub} can be proven in a different way, with no reference to the fractional chromatic number. Indeed, if $G_1,\dots,G_k$ is an $\F$-good family, while $T_1,\dots,T_m$ is a family satisfying the conditions of the dual problem, then all the symmetric differences of the form $G_i\oplus T_j$ are different, implying $km\le 2^{n\choose 2}$. This is true because if $G_i\oplus T_j$ and $G_r\oplus T_s$ would be the same for some $\{i,j\}\neq\{r,s\}$, then $(G_i\oplus T_j)\oplus(G_r\oplus T_s)$ would be the empty graph that could also be written (by commutativity and associativity of the symmetric difference) as $(G_i\oplus G_r)\oplus(T_j\oplus T_s)$. This would mean that $G_i\oplus G_r$ and $T_j\oplus T_s$ are two identical graphs. But if one of them is the empty graph, then the other cannot be empty and if both are nonempty, then one of them belongs to $\F$ while the other one does not, so this is impossible. $\Diamond$}
\end{remark}

\section{Global conditions}\label{global}

\subsection{Connectivity}

When we speak about the class of connected graphs in the following theorem, we mean
graphs with a single connected component, and hence without isolated vertices.
\begin{thm}~\label{thm:conn}
Let $\F_c$ denote the class of connected graphs. Then $$M_{\F_c}(n)=2^{n-1}.$$
\end{thm}

\proofka
First we give a very simple dual family $\B_c$. Let it consist of all graphs on $[n]$ in which the vertex labeled $n$ is isolated. Clearly $|\B_c|=2^{{n-1}\choose 2}$ (that is the number of all graphs on $[n-1]$) and $n$ is also isolated in the symmetric difference of any two of them, so no such symmetric difference is connected, This gives $D_{{\F_c}}(n)\ge 2^{{n-1}\choose 2}$ and thus by Lemma~\ref{lem:ub} we have
$$M_{\F_c}(n)\le 2^{{n\choose 2}-{n-1\choose 2}}=2^{n-1}.$$

Now we show that this upper bound can be attained.
Let the family $\A_c$ consist of all those graphs on $[n]$ that are the vertex-disjoint union of two complete graphs
(where each vertex belongs to one of them)
including the case when one of the two is on the empty set. Clearly, the number of these graphs is just half the number of subsets of $[n]$, that is exactly $2^{n-1}$. All we have to show is that the symmetric difference of any two of these graphs is connected. Choose two arbitrary graphs $G$ and $G'$ from our family. Let $G$ be the union of complete graphs on the complementary vertex sets $K$ and $L$, while $G'$ be the same on $K'$ and $L'$. Let $A=K\cap L', B=L'\cap L, C=L\cap K'$ and $D=K'\cap K$. It is possible that one, but only one of $A,B,C,D$ is empty. The edges of $G\oplus G'$ are all the edges of the complete bipartite graph with partite classes $A\cup C$ and $B\cup D$, so it must be connected.
\qed

With just a little more consideration one can also treat the case of $2$-connectedness at least for even $n$.

\begin{thm}~\label{thm:2conn}
Let $\F_{2c}$ denote the class of $2$-connected graphs. Then if $n$ is even, we have $$M_{\F_{2c}}(n)=2^{n-2}.$$
\end{thm}

\proofka
The proof is a modification of the previous one, therefore we use the notation introduced there. The construction given there may result in symmetric differences that are not $2$-connected only if $A\cup C$ or $B\cup D$ contains only one element. For even $n$ this can be avoided if we consider only such graphs in our construction where the bipartition of $[n]$ defining the individual graphs has an even number of elements in both partite classes $K$ and $L$. This proves the lower bound.

For the upper bound we consider all graphs in which the vertex $n$ is either isolated or it has one fixed neighbor, say $n-1$. The symmetric difference of any two such graphs is not $2$-connected, since $n$ has at most one neighbor in it. The number of such graphs
is just twice the number of graphs in which $n$ is an isolated point,
that is, $2^{{{n-1}\choose 2}+1}$ proving the matching upper bound by Lemma~\ref{lem:ub}.
\qed

\medskip
\par\noindent
\begin{remark}\label{rem:evenodd} {\rm The upper bound proven in Theorem~\ref{thm:2conn} clearly holds also for odd $n$ but we have not found a matching construction in general. For $n=3$ a triangle and the empty graph would do, still achieving the upper bound. But for larger odd $n$ the best we could do is to take only those graphs from our construction for which in the corresponding bipartition the smaller partition class has an odd number of elements if $n\equiv 1\ ({\rm mod}\ 4)$ and it has an even number of elements if $n\equiv 3\ ({\rm mod}\ 4)$. The number of graphs obtained this way is $2^{n-2}-{(n-2)\choose {(n-3)/2}}$. $\Diamond$}
\end{remark}

\medskip
\par\noindent
 \begin{remark}\label{rem:linear} {\rm Changing the graphs to their complements in the proofs of Theorems~\ref{thm:conn} and \ref{thm:2conn} makes these graph families vector spaces over the $2$-element field, while they still satisfy the conditions as the symmetric differences do not change by complementation (or by taking the symmetric difference of all elements with any fixed graph which is the complete graph in case of complementation). $\Diamond$}
\end{remark}

\medskip
\par\noindent
It does not sound surprising that if we step further on to $k$-connectedness for $k>2$ then the problem becomes rather more complicated. Nevertheless, if we insist on linear codes, that is graph families closed under the symmetric difference operation then for $k=3$ we can still determine the largest possible cardinality for infinitely many values on $n$ using Hamming codes.

\begin{thm} \label{thm:3connlin}
Let $\F_{3c}$ be the class of $3$-connected graphs and let $M^{(\ell)}_{\F_{3c}}(n)$ denote the size of a largest possible linear graph family on vertex set $[n]$ any two members of which have a $3$-connected symmetric difference. If $n=2^k-1$ for some integer $k\ge 2$, then
$$M^{(\ell)}_{\F_{3c}}(n)=2^{n-k-1}.$$
\end{thm}

\proofka
First we prove that $D_{{\F}_{3c}}(n)\ge n2^{{n-1\choose 2}}$ holds in general. Consider the family of all graphs on vertex set $[n]$ in which the degree of vertex $n$ is at most $1$. There are exactly $n2^{{n-1\choose 2}}$ such graphs. The symmetric difference of any two of these graphs is at most $2$-connected, since the vertex $n$ has degree at most $2$ in all these symmetric differences. This proves the claimed inequality and by Lemma~\ref{lem:ub} this implies
$M_{\F_{3c}}(n)\le 2^{n-1}/n$.
\smallskip
\par\noindent
It is well-known that if a family of subsets of a finite set contains the empty set and is closed under the symmetric difference operation then the cardinality of this set must be a power of $2$. This follows immediately from linear algebra and the fact that such a family forms a vector space over $GF(2)$, cf. e.g. Lemma 3.1 in Kozlov's book \cite{Kozlov} where a simple combinatorial proof of this fact is also presented. Since a linear graph family code on $[n]$ can be viewed as a collection of subsets of $E(K_n)$, this implies that $M^{(\ell)}_{\F_{3c}}(n)$ is a power of $2$. Since we obviously have $M^{(\ell)}_{\F_{3c}}(n)\le M_{\F_{3c}}(n)$, the upper bound proved above implies $M^{(\ell)}_{\F_{3c}}(n)\le 2^d$ with $d=\lfloor\log_2 \frac{1}{n}2^{n-1}\rfloor$ giving $$M^{(\ell)}_{\F_{3c}}(n)\le 2^{n-k-1}$$ for $n=2^k-1, k\ge 2$, which proves the required upper bound.
\smallskip
\par\noindent
For the lower bound consider the Hamming code ${\cal C}_H(n)$ with length $n=2^k-1$ that exists for every $k\ge 2$. (For a nice quick account on Hamming codes see e.g. \cite{Berlekamp}.) It is a linear code with minimum distance $3$ that consists of $2^{n-k}$ binary codewords having the property that if ${\mbf c}=(c_1,\dots,c_n)$ belongs to the code then so does also $\bar{\bf c}=(\bar c_1,\dots,\bar c_n)$ where $\bar c_i=1-c_i$. For each codeword ${\mbf c}\in {\cal C}_H(n)$ consider the bipartition of $[n]$ into the subsets $K_{\msbf c},L_{\msbf c}$, where $K_{\msbf c}=\{i: c_i=0\}, L_{\msbf c}=\{i:c_i=1\}$ and the complete bipartite graph $G_{K_{\msbf c},L_{\msbf c}}$ with partite classes $K_{\msbf c},L_{\msbf c}$. Note that by the above mentioned property of Hamming codes we have ${\mbf c}\in {\cal C}_H(n)$ if and only if $\bar{{\mbf c}}\in {\cal C}_H(n)$ and thus since
$G_{K_{{\msbf c}},L_{{\msbf c}}}=G_{K_{\bar{\msbf c}},L_{{\bar{\msbf c}}}}$, we get $\frac{1}{2}|{\cal C}_H(n)|=2^{n-k-1}$ different complete bipartite graphs this way. All we have to prove is that the symmetric difference of any two of our graphs is $3$-connected. This is equivalent to show that if ${\bf c}'\neq {\bf c},\bar{\bf c}$, then the cardinality of both partite classes of $G_{K_{\msbf c},L_{\msbf c}}\oplus G_{K_{{\msbf c}'},L_{{\msbf c}'}}$, that is of $(K_{\msbf c}\cap K_{{\msbf c}'})\cup (L_{\msbf c}\cap L_{{\msbf c}'})$ and $(K_{\msbf c}\cap L_{{\msbf c}'})\cup (K_{{\msbf c}'}\cap L_{\msbf c})$ is at least $3$. However, this immediately follows from the fact that the codeword ${\bf c}'$ must differ from both ${\bf c}$ and $\bar{{\bf c}}$ in at least $3$ coordinates. This completes the proof. \qed

\subsection{Hamiltonicity}

A graph is connected if and only if it contains a spanning tree. Next we consider what happens if we require the containment of specific spanning trees: a path in this subsection and a star in the next one.

\medskip
\par\noindent
\begin{thm}\label{thm:Hp}
Let $\F_{H{\rm p}}$ denote the class of graphs containing a Hamiltonian path. Then for infinitely many values of $n$ we have $$M_{\F_{H{\rm p}}}(n)=2^{n-1}.$$
In particular, this holds whenever $n=p$ or $n=2p-1$ for some odd prime $p$.
\end{thm}

\medskip
\par\noindent
To prove the above theorem we will refer to the following old conjecture that is known to be true in several special cases. To state it we need the notion of perfect $1$-factorization. It means the partition of the edge set of a graph into perfect matchings such that the union of any two of them is a Hamiltonian cycle.

\medskip
\par\noindent
{\bf Perfect $1$-factorization conjecture (P1FC)} (Kotzig~\cite{Kotzig}). {\em The complete graph $K_n$ has a perfect $1$-factorization for all even $n>2$.}

\medskip
\par\noindent
This conjecture is still open in general, however it is known to hold in several special cases, for example, whenever $n=p+1$ (Kotzig~\cite{Kotzig}) or $n=2p$ for some odd prime $p$ (Anderson~\cite{Anderson} and Nakamura~\cite{Nakamura}, cf. also Kobayashi~\cite{Kobayashi}). For a recent survey, see Rosa~\cite{Rosa}, according to which the smallest open case of the conjecture is $n=64$.

\medskip
\par\noindent
{\it Proof of Theorem~\ref{thm:Hp}.}
Since Hamiltonian paths are connected, it follows from the proof of Theorem~\ref{thm:conn} that $2^{n-1}$ is again an upper bound. Now we show that it is also a lower bound whenever the Perfect $1$-factorization conjecture holds for $n+1$. (Note that if the conjecture is true, then this means all odd numbers at least $3$, while for $1$ our statement is void.)

\medskip
\par\noindent
Let $n$ be an odd number for which $K_{n+1}$ has a perfect $1$-factorization $\M$ and $v$ a fixed vertex of $K_{n+1}$. Note that deleting the edge incident to $v$ from all matchings belonging to $\M$ we obtain $n$ matchings of $K_n$ such that the union of any two of them is a Hamiltonian path in $K_n:=K_{n+1}\setminus \{v\}$. Now consider all those subgraphs of $K_n$ that can be obtained as the union of an even number of these $n$ matchings. Clearly, the symmetric difference of any two of them is also the union of at least two of these matchings and thus contains a Hamiltonian path. The number of graphs obtained this way is $2^{n-1}$, matching the upper bound.
\qed

\medskip
\par\noindent
The case of Hamiltonian cycles can be treated essentially the same way.

\medskip
\par\noindent
\begin{thm}\label{thm:Hc}
Let $\F_{H{\rm c}}$ denote the class of graphs containing a Hamiltonian cycle. For all even values of $n$ for which the P1FC holds, we have $$M_{\F_{H{\rm c}}}(n)=2^{n-2}.$$ In particular, this is the case if $n=p+1$ or $n=2p$ for some odd prime $p$.
\end{thm}

\medskip
\par\noindent
\proofka
Since Hamiltonian cycles are $2$-connected, it follows from the proof of Theorem~\ref{thm:2conn} that $2^{n-2}$ is again an upper bound.

\medskip
\par\noindent
Let $n$ be an even number for which the P1FC holds and let $\M$ be a perfect $1$-factorization of $K_n$. Note that $\M$ contains $n-1$ matchings
(indeed the edge-chromatic number of $K_n$ for even $n$ is $n-1$). Now consider the $2^{n-2}$ graphs we can obtain as the union of an even number of matchings from $\M$. Clearly, the symmetric difference of any two of them contains a Hamiltonian cycle.
\qed

\medskip
\par\noindent
\begin{remark}\label{rem:borsuk} {\rm Since Hamiltonian cycles are $2$-connected graphs the proof of Theorem~\ref{thm:Hc} obviously gives an alternative proof of Theorem~\ref{thm:2conn} for those values of $n$ for which the Perfect $1$-factorization conjecture is known to hold. (The situation is similar for Theorems~\ref{thm:Hp} versus \ref{thm:conn}.) On the other hand, the construction in the proof of Theorem~\ref{thm:2conn} utterly fails to give a good lower bound for the value of $M_{\F_{Hc}}(n)$ investigated in Theorem~\ref{thm:Hc}. Indeed, the symmetric difference of two graphs in the construction given in the proof of Theorem~\ref{thm:2conn} contains a Hamiltonian cycle if and only if the sets denoted by $A\cup C$ and $B\cup D$ in that proof both have cardinality $\frac{n}{2}$ and this happens exactly when the partition classes of the partitions $(K,L)$ and $(K',L')$ are orthogonal in the sense 
that representing these bipartitions by characteristic vectors consisting of $+1$ and $-1$ coordinates in the obvious way, we get a collection of vectors that are pairwise orthogonal. So their number cannot be more than just $n$ and we can give exactly $n$ such vectors if and only if an $n\times n$ Hadamard matrix exists. $\Diamond$}
\end{remark}

\subsection{Containing a spanning star}

We have seen in the previous subsection that if we want every symmetric difference to contain a spanning tree which is a path, then for infinitely many values of $n$ our family can be just as large as if we did not want more than just the connectedness of these symmetric differences. In this subsection we show that if the required spanning tree is a star, then the largest possible family is drastically smaller.

\begin{thm}\label{thm:fullstar}
Let $\F_{S}$ denote the class of graphs containing a spanning star, that is a vertex connected to all other vertices in the graph. Then we have

$$M_{\F_S}(n)
=\left\{\begin{array}{lll}n+1&&\hbox{if }
n\ {\rm is\ odd}\ \\n&&\hbox{if }n\ {\rm is\ even}.\end{array}\right.$$
\end{thm}

\proofka
First we prove the upper bound. Let $G_1,\dots,G_m$ be an $\F_S$-good family on the vertex set $[n]$.
Consider the complete graph $K_m$ whose vertices are labeled with the graphs $G_1,\dots,G_m$.
For each edge $\{G_i,G_j\}$ of this graph assign an element $h\in [n]$ for which $h$ is adjacent to all other elements of $[n]$ in the graph $G_i\oplus G_j$. By the definition of $\F_S$-goodness such a $h$ exists for every pair of our graphs. Now observe that if an element $a\in [n]$ is assigned to two distinct edges $e$ and $f$ of our graph $K_m$, then $e$ and $f$ must be independent edges. Indeed, if that was not the case then we would have $e=\{G_i,G_j\}, f=\{G_i,G_k\}$ for some $i,j,k\in [n]$ and $a$ would be a full-degree vertex (that is one, connected to all other vertices) in both of the graphs $G_i\oplus G_j$ and $G_i\oplus G_k$. But since $G_j\oplus G_k=(G_i\oplus G_j)\oplus (G_i\oplus G_k)$, that would mean that $a$ is an isolated vertex in $G_j\oplus G_k$, so no vertex of this latter graph can have full degree contradicting the $\F_S$-goodness of our family. Thus our assignment of vertices from $[n]$ to the edges of our $K_m$ partitions the edge set of $K_m$ into sets of independent edges (every partition class consisting of the edges with the same assigned label), in other words, it defines a proper edge-coloring of $K_m$. This means that the number of possible labels, which is $n$, should be at least as large as the edge-chromatic number $\chi_e(K_m)$ of $K_m$. Since the latter is $m-1$ for even $m$ and $m$ for odd $m$, turning it around we obtain that for odd $n$ we must have $m\le n+1$ and for even $n$ we must have $m\le n$.

\medskip
\par\noindent
Now we show that the upper bound we proved is sharp. First assume that $n$ is odd and consider a complete graph $K_{n+1}$ on the vertices $v_1,\dots,v_{n+1}$ along with an optimal edge-coloring $c: E(K_{n+1})\to [n]$  of this graph. This edge-coloring partitions $E(K_{n+1})$ into $n$ disjoint matchings $M_1,\dots,M_n$, where $M_j$ consists of the edges colored $j$ for every $j\in [n]$. Now we construct the graphs $G_1,\dots,G_{n+1}$ by telling for each potential edge $ij$ of the complete graph on $[n]$ which $G_k$'s will contain it and which ones will not. Consider the edge $ij$ and the union of the matchings $M_i$ and $M_j$ (note that these matchings are in the ''other'' complete graph on $n+1$ vertices). This union is a bipartite graph on the vertex set $\{v_1,\dots,v_{n+1}\}$ with two equal size partite classes $A$ and $B$. Let $ij$ be an edge of the graph $G_k$ if and only if $v_k\in A$. (So $ij$ will be an edge of exactly half of our graphs $G_1,\dots,G_{n+1}$.) Do this similarly for all edges of $K_n$, the complete graph on vertex set $[n]$.
This way we defined our $n+1$ graphs. We have to show that they form an $\F_S$-good family.

To this end consider two of our graphs, say $G_h$ and $G_k$. The edge $\{v_h,v_k\}$ has got some color in our coloring $c$, call this color $j$. This means that $\{v_h,v_k\}$ belongs to the matching $M_j$. We claim this means that $j\in [n]$ is a full-degree vertex of $G_h\oplus G_k$. The latter is equivalent to the statement that every edge $ji$ incident to the point $j$ appears in exactly one of the graphs $G_h$ and $G_k$. But this follows from the way we constructed our graphs: when we decided about the edge $ji$ we considered the matchings $M_i$ and $M_j$ and the bipartite graph their union defines. Since $\{v_h,v_k\}\in M_j$, the points $v_h$ and $v_k$ are always in different partite classes of this bipartite graph, so whichever was called $A$, exactly one of $v_h$ and $v_k$ belonged to it. Thus the edge $ij$ was declared to be an edge of exactly one of $G_h$ and $G_k$. Since this is so for every $i\neq j$, $j$ is indeed a full-degree vertex in $G_h\oplus G_k$.

\medskip
\par\noindent
Assume now that $n$ is even. Then $n-1$ is odd and we can construct graphs $G_1,\dots,G_n$ on vertex set $[n-1]=\{1,\dots,n-1\}$ as given in the previous paragraph. These are not yet good, however, since we have an $n$th vertex that does not appear yet in any of the graphs. Note that we have $n-1$ matchings $M_1,\dots,M_{n-1}$ involved in the construction so far whose indices are just the first $n-1$ vertices of our graphs. Think about the additional vertex $n$ as the index of an additional ``matching'' $M_n$ that has no edges at all. We decide about the involvement of the edges $ni$ ($i<n$) in our graphs analogously as we did for the earlier edges: Consider the bipartite graph $M_i\cup M_n$, that consists of just the edges of $M_i$, so it is a perfect matching on the vertex set $\{v_1,\dots,v_n\}$. Let the two partite classes defined by this perfect matching be $A$ and $B$ and add the edge $ni$ to the graph $G_h$ if and only if $v_h$ belongs to $A$. Now we can prove analogously to the odd case that the symmetric difference of any two of our graphs contains a vertex of degree $n-1$. Consider $G_h$ and $G_k$. The edge between $v_h$ and $v_k$ in the auxiliary complete graph belongs to exactly one of the matchings $M_j$ and every edge $ij$ is in exactly one of the graphs $G_h$ and $G_k$ if $i\in\{1,\dots,j-1,j+1,\dots,n\}$. This completes the proof.
\qed

\medskip
\par\noindent
The only graph family code proven to be optimal and nonlinear (or not the coset of a linear code) in this paper is
the one appearing in the above Theorem~\ref{thm:fullstar}.
This is also the first case so far when the upper bound is proven without the use of Lemma~\ref{lem:ub}.
This suggests the question of what could be said about the dual problem in this case. The next theorem solves this dual problem for even values of $n$ also showing that Lemma~\ref{lem:ub} would not give a sharp upper bound for $M_{\F_S}(n)$.

\medskip
\par\noindent
\begin{thm} \label{thm:nostar}
If $n$ is even, then
$$D_{{\F}_S}(n)= 2^{{n\choose 2}-\frac{n}{2}}.$$
When $n$ is odd, then we have
$$2^{{n\choose 2}-\frac{n+1}{2}}\le D_{{\F}_S}(n)\le 2^{{n\choose 2}-\frac{n}{2}}.$$
\end{thm}

\medskip
\par\noindent
For the proof we will need the following celebrated
result from \cite{CFGS} (see also Corollary~15.7.7 in \cite{AS}).

\medskip
\par\noindent
{\bf Shearer's Lemma.} (\cite{CFGS})
{\it Let $S$ be a finite set and $A_1,\dots, A_m$ be subsets of $S$ such that every element of $S$ is contained in at least $k$ of the sets $A_1,\dots,A_m.$ Let $\M$ be a collection of subsets of $S$ and let $\M_i=\{T\cap A_i: T\in \M\}$ for $1\le i\le m$. Then
$$|\M|^k\le \prod_{i=1}^m |\M_i|.$$}

\medskip
\par\noindent
{\em Proof of Theorem~\ref{thm:nostar}.}
We will prove $$2^{{n\choose 2}-\lceil\frac{n}{2}\rceil}\le D_{{\F}_S}(n)\le 2^{{n\choose 2}-\frac{n}{2}}$$ that implies both the even and the odd case.
For the lower bound fix a subgraph $T$ of $K_n$ with the minimum number $\lceil\frac{n}{2}\rceil$ of edges such that no vertex is isolated and take all possible subgraphs of $K_n$ that contain all edges of $T$. The number of such subgraphs is $2^{{n\choose 2}-\lceil\frac{n}{2}\rceil}$ and no two of them has a symmetric difference that contains all edges incident to any fixed vertex. This proves the lower bound.
\smallskip
\par\noindent
For the upper bound consider a graph family $\M$ that satisfies the condition that no two of its elements have a symmetric difference with a vertex of degree $n-1$. For $i=1,\dots,n$ let $S_i$ be the set of $n-1$ edges (of $K_n$) incident to vertex $i$. Then for any $T, T'\in \M$ we cannot have $E(T')\cap S_i=S_i\setminus (E(T)\cap S_i),$
that is, $E(T)$ and $E(T')$ cannot be complementary on any $S_i$. So if $\M_i$ denotes the family of graphs obtained by taking the projection of all graphs from $\M$ to the edge set $S_i$, then $|\M_i|\le 2^{n-2}.$ Since each edge of $K_n$ appears in exactly two of the sets $S_i$, we can apply Shearer's Lemma to these sets with $k=2$. This gives $$|\M|^2\le \prod_{i=1}^n |\M_i|\le 2^{n(n-2)}.$$ Taking square roots we get the upper bound.
\qed

\medskip
\par\noindent
Note that if we restrict attention to linear graph families for the dual problem treated in Theorem~\ref{thm:nostar}, then using again that the cardinality of such a family should be a power of $2$ (cf. the similar argument in the proof of Theorem~\ref{thm:3connlin}) we get that our lower bound is also sharp for odd values of $n$.

\section{Local conditions}\label{local}

In the previous section we investigated $M_{\F}(n)$ in cases when the required symmetric differences contain specific spanning subgraphs, therefore to check whether these conditions are satisfied we have to consider our graphs on the whole vertex set. Now we turn to families $\F$
defined by containing some fixed small finite graphs, so the
nature of these conditions will be local.

\subsection{General local conditions}

\begin{defi}\label{defi:local}
A graph class $\LL$ defines a {\em local condition} if it has the property that whenever $H_1$ is an induced subgraph of $H_2$ and $H_1$ belongs to $\LL$ then so does also $H_2$. In short, we will refer to such an $\LL$ as a {\em local graph class}.
\end{defi}

\medskip
\par\noindent
Note that the above definition implies that whenever two graphs $F$ and $G$ are in the $\LL$-good relation (that is, $F\oplus G\in\LL$) then any $F'$ with $F'[U]\cong F$ and $G'$ with $G'[U]\cong G$ for some $U\subseteq V(F')=V(G')$ (that is, $F'$ and $G'$ induce subgraphs isomorphic to $F$ and $G$, respectively, on the same subset $U$ of their vertex set) are also in the $\LL$-good relation. This means that if two graphs are in this relation then there is always some local certificate for this.

\medskip
\par\noindent
Here are some examples of local graph classes.
\smallskip
\par\noindent
1. $\LL=\{H: L\subseteq H\}$ for some fixed finite simple graph $L$. That is $\LL$ contains all graphs that contain a (not necessarily induced) subgraph isomorphic to $L$. When $\LL$ is such a family we will use the simplified notation $M_L(n)$ for $M_{\LL}(n)$.
\smallskip
\par\noindent
2. $\LL=\{H: L\subseteq_{ind} H\}$ for some fixed finite simple graph $L$. That is $\LL$ contains all graphs that
have
an induced subgraph isomorphic to $L$. When $\LL$ is such a family we will use the simplified notation $M_{L,{\rm ind}}(n)$ for $M_{\LL}(n)$.
\smallskip
\par\noindent
(Note that although the above two examples give different notions, the word ``induced'' is indeed needed in Definition~\ref{defi:local}.)
\smallskip
\par\noindent
3. $\LL=\C_{\rm odd}:=\{H: C_{2k+1}\subseteq H\ {\rm for\ some\ integer}\ 1\le k\}$, that is, $\C_{\rm odd}$ contains all graphs that contain
an odd cycle.
\smallskip
\par\noindent
4. For some fixed integers $h$ and $\ell$ we can define $\LL_{h,\ell}=\{H: \exists U\subseteq V(H), |U|=h, |E(H[U])|=\ell\}$,
that is, $\LL_{h,\ell}$ is the class of all graphs that have an induced subgraph on $h$ vertices with exactly $\ell$ edges.

\medskip
\par\noindent
In the following we prove some general results related to $M_{\LL}(n)$ for local graph classes $\LL$ and will further investigate the special
case belonging to our first example above in the next subsection.
In Subsection~\ref{subsect:triodd} we will focus on $M_{K_3}(n)$
and $M_{\C_{\rm odd}}(n)$. In the final subsection we discuss the
behaviour of the functions
$M_{L,{\rm ind}}(n)$ mentioned in the second example above.

\medskip
\par\noindent
The next proposition gives a straightforward upper bound on the value of $M_{\LL}(n)$. It is in terms of $ex(n,\LL)$ that, as usually in extremal graph theory, denotes the maximum number of edges a graph on $n$ vertices can have without containing any $L\in \LL$ as a subgraph.

\medskip
\par\noindent
\begin{prop}\label{thm:Wilex}
For any local graph class $\LL$
$$M_{\LL}(n)\le 2^{{n\choose 2}-ex(n,\LL)}.$$
\end{prop}

\proofka
Consider an $n$-vertex graph $H$ satisfying $|E(H)|=ex(n,\LL)$ and containing no subgraph isomorphic to any $L\in\LL$. The family of all subgraphs of $H$ clearly satisfies the requirements of the dual problem of $M_{\LL}(n)$ (no two graphs in that family can have a symmetric difference containing some $L\in \LL$) and the family has size $2^{ex(n,\LL)}$. Thus the claimed upper bound follows from Lemma~\ref{lem:ub}. \qed

\medskip
\par\noindent
Proposition~\ref{thm:Wilex} and our following results will justify the relevance of the following notion in our current setting.

\begin{defi}~\label{defi:rate}
The rate $R_{\LL}(n)$ of an optimal graph family code on $n$ vertices satisfying the requirement prescribed by the local graph class $\LL$ is defined as
$$R_{\LL}(n):=\frac{2}{n(n-1)}\log_2M_{\LL}(n).$$
\end{defi}

\medskip
\par\noindent
We will soon see that the value $\limsup_{n\to\infty}R_{\LL}(n)$ is strictly positive for any $\LL$ belonging to this section.
We will use the following theorem due to Wilson to show that the limit actually exists for all local graph classes.

\medskip
\par\noindent
{\bf Wilson's theorem.} (\cite{Wilson})
{\it For every finite simple graph $T$ there exists a threshold $n_0(T)$ such that if $n>n_0(T)$ and the following two conditions hold then the edge set of the complete graph $K_n$ can be partitioned into subgraphs each of which is isomorphic to $T$. The two conditions are:
\par\noindent
1. $n\choose 2$ is divisible by $|E(T)|$;
\par\noindent
2. $n-1$ is divisible by the greatest common divisor of the degrees of vertices in $T$.}

\medskip
\par\noindent
Note that the two conditions in the above theorem are obviously necessary. The decomposition of $K_n$ in the conclusion of the theorem is called a $T$-design when it exists, cf. \cite{ABB}.

\begin{thm}\label{thm:limes}
Let $\LL$ be an arbitrary fixed local graph class. Then the value
$\lim_{n\to\infty}R_{\LL}(n)$ exists and is bounded from below by $R_{\LL}(n)$ for every $n$.
\end{thm}

\proofka
Let $n$ be an arbitrary natural number and let
${\cal G}=\{G_1,\dots,G_m\}$ be an optimal graph family code for $\LL$ with $V(G_i)=[n], i\in\{1,\dots,m\}$, that is one with $m=M_{\LL}(n)$. By Wilson's theorem a $K_n$-design exists for $K_N$, whenever $N$ is large enough and both $n-1$ divides $N-1$ and $n\choose 2$ divides $N\choose 2$. Take such an $N$ and consider the $K_n$-design on $K_N$ consisting of the subgraphs $K^{(1)},\dots,K^{(r)}$, where $r=\frac{N(N-1)}{n(n-1)}$ and each $K^{(i)}$ is isomorphic to $K_n$. Now let ${\cal G}_j:=\{G_1^{(j)},\dots,G_m^{(j)}\}$ be an optimal graph family code for $\LL$ on $V(K^{(j)})$ for every $j\in\{1,\dots,r\}$. (Obviously, we can choose each ${\cal G}_j$ to be isomorphic to ${\cal G}$.) Now define a graph family code on $K_N$ for $\LL$ as the collection of graphs that can be written in the form of $G_{\msbf a}:=\cup_{j=1}^r G_{a_j}^{(j)}$ where ${\mbf a}=(a_1,\dots,a_r)$ runs through all possible sequences satisfying $a_i\in\{1,\dots,m\}$ for every $i$. Since there are $m^r$ such sequences ${\mbf a}$, this way we have $m^r$ different graphs in our family.
They form indeed a graph family code for $\LL$ since for any two of them,
$G_{\msbf a}$ and $G_{\msbf b}$ there is some $j$ for which
$a_j\neq b_j$ and thus $G_{\msbf a}\oplus G_{\msbf b}
\supseteq_{ind} G_{a_j}\oplus G_{b_j}\supseteq_{ind} L$ for some $L\in\LL$.
This implies $M_{\LL}(N)\ge m^r$ and thus $$R_{\LL}(N)\ge \frac{2}{N(N-1)}\log_2m^r=\frac{2}{n(n-1)}\log_2 M_{\LL}(n)=R_{\LL}(n).$$
\smallskip
\par\noindent
The requirements for $N$ are satisfied if $N=kn(n-1)+1$ and $k$ is large enough. (Also for $N=kn(n-1)+n$ and large enough $k$ but considering the former is enough for our argument.) Since $M_{\LL}(n)$ is clearly monotone nondecreasing in $n$ (as we can always ignore some vertices and consider a graph family code only on the rest), we can write that for any $kn(n-1)+1\le i\le (k+1)n(n-1)$ we have $M_{\LL}(i)\ge m^r$ for $r=\frac{{{kn(n-1)+1}\choose 2}}{{n\choose 2}}$. Introducing the sequence $b_i:= m^r$ for $r=\frac{{{kn(n-1)+1}\choose 2}}{{n\choose 2}}$ whenever $kn(n-1)+1\le i\le (k+1)n(n-1)$ we can write $$\liminf_{i\to\infty}\frac{2}{i(i-1)}\log_2 M_{\LL}(i)\ge \liminf_{i\to\infty}\frac{2}{i(i-1)}\log_2 b_i\ge$$ $$\liminf_{k\to\infty}\frac{1}{{{{(k+1)n(n-1)}}\choose 2}}\log_2 m^{\frac{{{kn(n-1)+1}\choose 2}}{{n\choose 2}}}=\liminf_{k\to\infty}\frac{{{kn(n-1)+1}\choose 2}}{{{{(k+1)n(n-1)}}\choose 2}}\frac{2}{n(n-1)}\log_2m=R_{\LL}(n).$$
This proves that $\lim_{n\to\infty} R_{\LL}(n)$ exists and is equal to $\sup_n R_{\LL}(n)$.
\qed

\medskip
\par\noindent
\begin{remark}\label{rem:fekete}
{\rm The above proof is similar to proving that the limit defining the Shannon capacity of graphs exists which is usually done using Fekete's Lemma. Here, however, there are some technical subtleties (because of the divisibility requirements for $N$) that made it simpler to present a full proof than to refer simply to Fekete's Lemma. $\Diamond$}
\end{remark}

\medskip
\par\noindent
In view of Theorem~\ref{thm:limes} the following definition is meaningful.

\smallskip
\par\noindent
\begin{defi}~\label{defi:dc}
The distance capacity (or {\em distancity} for short) of a local graph class $\LL$ is defined as $$DC(\LL):=\lim_{n\to\infty} R_{\LL}(n).$$
\end{defi}

\medskip
\par\noindent
Based on Tur\'an's celebrated theorem \cite{Turan} (cf. also e.g. in \cite{Diestel}) and the famous theorem of Erd\H{o}s and Stone~\cite{ErdStone}, Erd\H{o}s and Simonovits \cite{ErdSim} proved that if $\LL$ is an arbitrary family of graphs, then
\begin{equation}\label{eq:EStS}
\lim_{n\to\infty}\frac{ex(n,\LL)}{{n\choose 2}}=1-\frac{1}{\chi_{\rm min}(\LL)-1},
\end{equation}
where $\chi_{\rm min}(\LL)=\min_{L\in\LL}\chi(L)$ and $\chi(G)$ denotes the chromatic number of graph $G$.
(We assume that $\chi_{\rm min}(\LL)\ge 2$. For the case when $\LL$ contains some edgeless graph see Remark~\ref{rem:empty}.)

\medskip
\par\noindent
Note that Proposition~\ref{thm:Wilex} and the above result determining the order of magnitude of $ex(n,\LL)$ has the following immediate consequence for the distancity.

\medskip
\par\noindent
\begin{cor}\label{cor:DCub}
For any local graph class $\LL$ with $\chi_{\rm min}(\LL)\ge 2$ we have
$$DC(\LL)\le \frac{1}{\chi_{\rm min}(\LL)-1}.$$
\qed
\end{cor}

\subsection{Containing a prescribed subgraph}\label{subsect:noninduced}

\medskip
\par\noindent
Now we focus on local graph classes mentioned in our first example after Definition~\ref{defi:local}: we have some fixed finite simple graph $L$ and
consider $\LL=\{H: L\subseteq H\}$. As said above in this case we will use the notation $M_L(n)$ for $M_{\LL}(n)$ and similarly, we will also denote $R_{\LL}(n)$ and $DC(\LL)$ by $R_L(n)$ and $DC(L)$, respectively. We prove that in this case the upper bound of Corollary~\ref{cor:DCub} is always sharp.

\medskip
\par\noindent
\begin{thm}\label{thm:DCL}
For any fixed graph $L$ we have $$DC(L)=\frac{1}{\chi(L)-1}.$$
\end{thm}

\medskip
\par\noindent
For the proof we will use a result by Erd\H{o}s, Frankl and R\"odl \cite{EFR} about the number $F_n(L)$ of graphs on $n$ labeled vertices containing no subgraph isomorphic to $L$.

\medskip
\par\noindent
{\bf Erd\H{o}s-Frankl-R\"odl theorem.} (\cite{EFR})
{\it Suppose $\chi(L)=r\ge 3$. Then
$$F_n(L)=2^{ex(n,K_r)(1+o(1))}.$$}

\par\noindent
Note that this gives
$$F_n(L)=2^{{n\choose 2}\left(1-\frac{1}{\chi(L)-1}+o(1)\right)}$$
by (\ref{eq:EStS}) (in fact, already directly by Tur\'an's theorem).
\medskip
\par\noindent
While the proof of the Erd\H{o}s-Frankl-R\"odl theorem is based on Szemer\'edi's Regularity Lemma, a similar result for bipartite $L$
easily follows from (\ref{eq:EStS}) (or from the K\H{o}v\'ari-S\'os-Tur\'an
Theorem \cite{KST}).
Indeed, it implies that if $L$ is bipartite then $F_n(L)<{{n\choose 2}\choose {\varepsilon {n\choose 2}}}$ for any $\varepsilon>0$
provided $n>n_0(\eps)$,
and that implies the claimed statement. (To see the latter one can use the well-known fact, cf. e.g. Lemma 2.3 in \cite{CsK}, that $${t\choose {\alpha t}}=2^{t(h(\alpha)+o(1))},$$ where $h(x)=-x\log_2x-(1-x)\log_2(1-x)$ is the binary entropy function and $0\le \alpha\le 1$ is meant to be such that $\alpha t$ is an integer. Applying this for $t:={n\choose 2}$ and $\alpha=\varepsilon$ we obtain that for any $0<\varepsilon<1$ the number ${{n\choose 2}\choose {\varepsilon {n\choose 2}}}$ is more than $2^{\delta{n\choose 2}}$ for some positive $\delta$.)

\medskip
\par\noindent
{\it Proof of Theorem~\ref{thm:DCL}.}
It follows immediately from Corollary~\ref{cor:DCub} that the right hand side is an upper bound on the left hand side so we only have to prove the reverse inequality.
\smallskip
\par\noindent
To this end let $G_L$ denote the graph whose vertices are all possible graphs on $n$ labeled vertices and two are connected if and only if their symmetric difference does not contain $L$ as a subgraph. (Note that this is just the complementary graph of $H_{\F}$ used in the proof of Lemma~\ref{lem:ub} when $\F$ is set to be the local graph class $\LL$ belonging to our problem.) Then $M_L(n)$ is equal to the independence number $\alpha(G_L)$ of $G_L$.
Clearly, $G_L$ is vertex-transitive (cf. the argument in the proof of Lemma~\ref{lem:ub} for $H_{\F}$), in particular, it is regular. Since the degree of its vertex representing the edgeless graph is just $F_n(L)$, we get (denoting the maximum degree of a graph $G$ by $\Delta(G)$) that
$$M_L(n)=\alpha(G_L)\ge\frac{|V(G_L)|}{\Delta(G_L)+1}=\frac{|V(G_L)|}{F_n(L)+1}=\frac{2^{{n\choose 2}}}{2^{{n\choose 2}\left(1-\frac{1}{\chi(L)-1}+o(1)\right)}}=2^{{n\choose 2}\left(\frac{1}{\chi(L)-1}+o(1)\right)}$$
by the Erd\H{o}s-Frankl-R\"odl theorem (and by the above discussion also for bipartite graphs). Putting this inequality into the definition of $DC(L)$ the required result follows.
\qed

\medskip
\par\noindent
\begin{cor}\label{DCgen}
Let $\G$ be a set of graphs, each containing at least one edge,
and let $\LL_{\G}$ be the local graph class containing all graphs that contain at least one $G\in\G$ as a subgraph. Then $$DC(\LL_{\G})=\frac{1}{\chi_{\rm min}(\LL_{\G})-1}=\frac{1}{\chi_{\rm min}(\G)-1}.$$ In particular, $$DC(\C_{\rm odd})=DC(K_3)=\frac{1}{2}.$$
\end{cor}

\proofka
The second statement is clearly a special case of the first one, so it is enough to prove the latter. It is a straightforward consequence of Corollary~\ref{cor:DCub} that the left hand side is bounded from above by the right hand side. For the reverse inequality note the trivial fact that $DC(\LL_{\G})\ge DC(G)$ for any $G\in\G$. Applying this for some $G\in\G$ that
satisfies $\chi(G)=\min_{G\in\G}\chi(G)=\chi_{\rm min}(\LL_{\G})$ the statement follows from Theorem~\ref{thm:DCL}.
\qed

\medskip
\par\noindent
\begin{remark}\label{rem:asydual}
{\rm It is straightforward from the foregoing that the above results also determine for any graph family $\G$ the asymptotic behaviour of the value $D_{\LL_{\G}}(n)$ belonging to the dual problem. Indeed, by Lemma~\ref{lem:ub} and Corollary~\ref{DCgen} we have that
$\lim_{n\to\infty}\frac{1}{{n\choose 2}}\log D_{\LL_{\G}}(n)\le 1-DC(\LL_{\G})=1-\frac{1}{\chi_{\min}(\G)-1}$ while a matching lower bound follows from the argument in the proof of Proposition~\ref{thm:Wilex}. Thus we have
$$\lim_{n\to\infty}\frac{2}{n(n-1)}\log D_{\LL_{\G}}(n)=1-\frac{1}{\chi_{\min}(\G)-1}$$
for any graph family $\G$. This means that by taking all subgraphs of a graph with the largest possible number of edges without containing a subgraph from $\G$ we obtain asymptotically a largest family of graphs no two of which have any $G\in\G$ in their symmetric difference.
$\Diamond$}
\end{remark}

\subsection{Containing a triangle or an odd cycle} \label{subsect:triodd}

In this subsection we are investigating $M_L(n)$ for small values of $n$ and the simplest $3$-chromatic graph, which is the triangle $K_3$. We will also look at the analogous problem when $K_3$, the cycle of length 3 is replaced by the family of all odd cycles.

\medskip
\par\noindent
For $L=K_3$ the bound of Proposition~\ref{thm:Wilex} gives us $M_{K_3}(n)\le 2^{{{n\choose 2}-\lceil\frac{n}{2}\rceil\lfloor\frac{n}{2}\rfloor}}$. Below we show that this upper bound is tight whenever $n$ is at most $6$.

\medskip
\par\noindent
The first part of the following Proposition is very simple and we present it only for the sake of completeness.

\medskip
\par\noindent
\begin{prop}~\label{prop:triv34}
We have $M_{K_3}(3)=2$ and $M_{K_3}(4)=4$.
\end{prop}

\proofka
For $n=3$ the statement is trivial: take the empty graph and a triangle on three vertices, this $2$-element family already achieves the value of the upper bound which is $2$ for $n=3$.

\medskip
\par\noindent
For $n=4$ we give the following four graphs on the vertex set $\{1,2,3,4\}$ by their edge sets.
Let $$E(G_0)=\emptyset, E(G_1)=\{12,23,13,34\}, E(G_2)=\{23,34,24,14\}, E(G_3)=\{12,13,24,14\}.$$
It takes an easy checking that the symmetric difference of any two of these graphs contains a triangle. Since the upper bound in Proposition~\ref{thm:Wilex} is also $4$ in this case, this proves that $M_{K_3}(4)=4$.
\qed

\medskip
\par\noindent
\begin{remark}\label{rem:lspace} {\rm Note that both of the above simple constructions are closed under the symmetric difference operation, that is they form a linear space over $GF(2)$ when the graphs are represented by the characteristic vectors of their edge sets. In fact, the second construction could also be presented as the vector space generated in this sense by any two of the graphs $G_1, G_2, G_3$. $\Diamond$}
\end{remark}

\medskip
\par\noindent
\begin{prop}\label{lem:MK3_5}
$$M_{K_3}(5)=16.$$
\end{prop}

\proofka
The value of the upper bound in Proposition~\ref{thm:Wilex}
gives $16$ for $n=5$, so we only have to prove that $16$ is also a lower bound.
To this end we will give a set of graphs forming a vector space in the sense of Remark~\ref{rem:lspace}. We will give this vector space by a set of generators, although in a somewhat redundant way. (Our reason to keep this redundancy is that the construction has more symmetry this way.)

\smallskip
\par\noindent
Think about the vertices $\{1,2,3,4,5\}$ as if they were given on a circle at the vertices of a regular pentagon in their natural order. Consider the graph with edge set $$E(G_1):=\{12,23,13,35\}.$$ Let $G_2, G_3, G_4, G_5$ be the four graphs we obtain from $G_1$ by rotating it along the circle containing the vertices so that vertex $1$ moves to $2$, $2$ to $3$, etc. Thus we have $$E(G_2)=\{23,34,24,41\}, E(G_3)=\{34,45,35,52),$$ $$E(G_4)=\{45,51,41,13\}, E(G_5)=\{51,12,52,24\}.$$
Now we consider the linear space the characteristic vectors of the edge sets of these five graphs $G_i, i\in\{1,2,3,4,5\}$ generate. These graphs can be defined as the elements of the family ${\cal G}=\{G_I: I\subseteq [5]\},$ where $$G_I=\oplus_{i\in I} G_i,$$ meaning that $V(G_I)=[5]$ and $E(G_I)$ contains exactly those edges that appear in an odd number of the graphs $G_i$ with $i\in I$.
\medskip
\par\noindent
Note that every edge of the underlying $K_5$ on $[5]$ appears in exactly two of the graphs $G_1,\dots,G_5$, therefore for $I=[5]$ we have that $G_I$ is the empty graph just as $G_{\emptyset}$ is. This implies that for every $I\subseteq [5]$ and $\overline{I}:=[5]\setminus I$ we have $G_I=G_{\overline{I}}$, thus every graph in our graph family has exactly two representations as $G_I$ for some $I\subseteq [5]$. (The two representations are given by $I$ and $\overline{I}$ as we have seen. It also follows that if $J\neq I,\overline{I}$ then $G_J\neq G_I$,
otherwise we would have $G_{J\oplus I}$ be the empty graph for
$J\oplus I\notin\{\emptyset, [5]\}$ contradicting that every edge appears exactly twice in the sets $E(G_i),\ i=1,\dots,5$.) Thus we have indeed $\frac{1}{2}2^5=16$ graphs in our family matching our upper bound for $n=4$.
\medskip
\par\noindent
We have to show that the symmetric difference of any two of our graphs contains a triangle. Since our construction is closed for the symmetric difference operation this is equivalent to say that all graphs in our family except the empty graph contains a triangle. Since $G_I=G_{\overline{I}}$ it is enough to prove that $G_I$ contains a triangle for all $1\le |I|\le 2, I\subseteq [5]$. This is easy to see when $|I|=1$. For subsets with $|I|=2$ it is enough to check this for $I=\{1,2\}$ and $I=\{1,3\}$ by the rotational symmetry of our construction. But these two cases are easy to check: $G_{\{1,2\}}$ contains the triangles on the triples of vertices $1,2,4$ and $1,3,4$, while $G_{\{1,3\}}$ contains the triangle on vertices $1,2,3$.
\qed

\medskip
\par\noindent
\begin{prop}\label{lem:MK3_6}
$$M_{K_3}(6)=64.$$
\end{prop}

\proofka
The value of the upper bound given by Proposition~\ref{thm:Wilex} is $2^6$ for $L=K_3$ and $n=6$, so we need to prove only the lower bound.
\smallskip
\par\noindent
To this end we give a construction of $64$ graphs forming a graph family code on $[6]$ for $K_3$. The construction will have several similarities to that in Proposition~\ref{lem:MK3_5} though with somewhat less symmetry. But again our graphs will form a vector space in the sense of Remark~6 to be specified through a set of seven generators that altogether cover each one of the edges of the underlying $K_6$ exactly twice, so every member of our graph family will have exactly two representations by the generators just as in the proof of Proposition~\ref{lem:MK3_5}. Here are the details.
\smallskip
\par\noindent
Think about the $6$ vertices $1,\dots,6$ as being on a circle in the vertices of a regular hexagon in their natural order as we go around the circle. Our first four generator graphs are the following four edge-disjoint triangles (plus three isolated points) given by their edge sets as follows.
$$E(G_1)=\{12,23,13\}, E(G_2)=\{34,45,35\}, E(G_3)=\{56,16,15\}, E(G_4)=\{24,46,26\}.$$
The other three graphs are three $K_4$'s (plus two isolated vertices) that are rotations of each other, in particular,
$$E(G_5)=\{12,24,45,15,14,25\}, E(G_6)=\{23,35,56,26,25,36\},$$ $$E(G_7)=\{34,46,16,13,36,14\}.$$
It is easy to check that the above seven graphs cover each edge of the underlying $K_6$ exactly twice. Just as in the proof of Proposition~\ref{lem:MK3_5} this implies that the generated family of graphs of the form $$G_I=\oplus_{i\in I}G_i$$ where $I$ runs through all subsets of $[7]$ contains exactly two representations of this form for each of its members, namely $$G_I=G_J\ {\rm if\ and\ only\ if}\ J=[7]\setminus I.$$
Thus our family has $2^6=64$ members that matches our upper bound. Now we have to show that the symmetric difference of every pair of our graphs contains a triangle. Since the family is closed under symmetric difference this is equivalent to every $G_I$ except $G_{\emptyset}=G_{[7]}$ containing a triangle. To show this we consider the representation of each of our graphs as $G_I$ where $I$ contains at most one of the three $K_4$ generators, that is $|I\cap \{5,6,7\}|\le 1$. When $I\cap \{5,6,7\}=\emptyset$ but $I$ itself is nonempty then this is trivial as in such a case $G_I$ is the union of some of the edge-disjoint graphs $G_1,\dots,G_4$ each of which is a triangle itself. In case $|I\cap \{5,6,7\}|=1$, then by symmetry we may assume w.l.o.g. that $I\cap \{5,6,7\}=\{5\}$. Then if we also have $\{1,2\}\subseteq I$ then the triangles on vertices $1,3,4$ and $2,3,5$ (and two more) will be contained in $G_I$. So we may assume that at least one of $G_1$ and $G_2$ is not part of our representation of $G_I$ and by symmetry, we may assume $2\notin I$. But then to avoid the triangles on vertices $1,4,5$ and $2,4,5$ being in $G_I$ we need both $3\in I$ and $4\in I$. In this case, however, we will have the triangle on vertices $4,5,6$ present in $G_I$. This completes the proof. \qed

\medskip
\par\noindent
Recall ${\cal C}_{\rm odd}$ be the class of all graphs containing an odd cycle. Since $ex(n,\C_{\rm odd})=ex(n,K_3)$ the upper bound of Proposition~\ref{thm:Wilex} is also $2^{{n\choose 2}-\lceil\frac{n}{2}\rceil\lfloor\frac{n}{2}\rfloor}$ for $M_{{\cal C}_{\rm odd}}(n)$. Since $K_3\cong C_3$ is an odd cycle, we obviously have $M_{K_3}(n)\le M_{{\cal C}_{\rm odd}}(n)$ and so by Propositions~\ref{prop:triv34}, \ref{lem:MK3_5} and \ref{lem:MK3_6} the previous upper bound is also sharp for $M_{{\cal C}_{\rm odd}}(n)$ when $n\in\{3,4,5,6\}$.
Although we could not prove that $M_{K_3}(7)$ is also equal to this upper bound, we can show this at least for $M_{{\cal C}_{\rm odd}}(7)$.

\medskip
\par\noindent
\begin{prop}\label{lem:MC7}
$$M_{{\cal C}_{\rm odd}}(7)=2^9.$$
\end{prop}

\proofka
The upper bound $2^{{n\choose 2}-\lceil\frac{n}{2}\rceil\lfloor\frac{n}{2}\rfloor}$ is equal to $2^9$, so it is enough to prove that this is also a lower bound. This we do similarly as in the proofs of Propositions~\ref{lem:MK3_5} and \ref{lem:MK3_6}.
\smallskip
\par\noindent
Again, we think about the seven vertices forming the set $[7]$ as the vertices of a regular $7$-gon around a cycle in their natural order. We define $7+3=10$ simple graphs $G_1,\dots,G_7$ and $G_8,\dots,G_{10}$ that will generate our family.
Let $G_1$ be the triangle with edges $12,24,14$ and $G_2,\dots,G_7$ be its six possible rotated versions, that is the triangles with edge sets $\{23,35,25\}, \{34,46,36\},\dots,\{17,13,37\}$, respectively. Note that these seven triangles cover all pairs of vertices exactly once, that is, they form a Steiner triple system. The three other graphs $G_8,G_9,G_{10}$ are three edge-disjoint seven-cycles, namely those with edge sets $$\{12,23,34,45,56,67,17\}, \{13,35,57,27,24,46,16\}, \{14,47,37,36,26,25,15\},$$ respectively. Note that these three graphs also cover all pairs of vertices exactly once and that the edge sets of a $G_i$ for $i\in [7]$ and  $G_j$ with $j\in \{8,9,10\}$ intersect in exactly one element.
Since our ten graphs cover the edges of the underlying $K_7$ exactly twice, just as in the proofs of Propositions~\ref{lem:MK3_5} and \ref{lem:MK3_6} the generated family $$\oplus_{i\in I}G_i$$ as $I$ runs over all subsets of $\{1,\dots,10\}$ will have exactly $2^9$ distinct members each of which is represented by two subsets of $\{1,\dots,10\}$, some $I$ and its complement. All we are left to show for proving $M_{{\cal C}_{odd}}(7)\ge 2^9$
is that each such $G_I$ except $G_{\emptyset}=G_{[10]}$ contains an odd cycle. If $I\subseteq [7]$, this is obvious and so is also if $I\subseteq\{8,9,10\}$. When both $I\cap [7]$ and $I\cap\{8,9,10\}$ are nonempty, then we consider that representation $G_I$ which has $|I\cap [7]|\le 3$. If we have $|I\cap\{8,9,10\}|=1$ then whichever $7$-cycle we have (that is, whichever of $G_8,G_9,G_{10}$) it will have two consecutive edges that do not appear in either of the at most three triangles. If we take the first pair of such edges (as we go along our $7$-cycle in an appropriate direction) for which the previous one is an edge of one of our triangles (since we take at least one triangle and each triangle intersects each $7$-cycle, such an edge must exist), then the construction ensures that these two consecutive edges close up to a $K_3$ in our $G_I$. In case we have two $7$-cycles in our $G_I$ representation, then those create $7$ distinct $K_3$'s in their union. Each of our triangles
intersects exactly three of those seven $K_3$'s created, so if we have $|I\cap [7]|\le 2$ then at least one of these seven $K_3$'s remain untouched. Thus we are left with the case of two $7$-cycles and exactly three triangles. For this case let us switch to the complementary representation with four triangles and one $7$-cycle. By symmetry, we may assume that our $7$-cycle is $G_8$. If the four triangles are such that two consecutive edges of $G_8$ do not appear in any of them then we can finish the argument as before. If this is not the case, then the four triangles must leave three such edges of $G_8$ uncovered which form a matching. Because of symmetry we may assume that these are the edges $12, 34, 56$. This also tells us exactly which are the four triangles we have in the representation of $G_I$, namely those that contain the remaining four edges, that is, $G_2, G_4, G_6$ and $G_7$. In this case $G_I$ contains the $K_3$,  for example, on the vertices $2,5,6$. Finally, if we have all the three $7$-cycles in our representation then the complementary representation has no $7$-cycle at all and this case we have already covered. This completes the proof. \qed

\subsection{Containing a prescribed induced subgraph}\label{subsect:induced}

\medskip
\par\noindent
In this subsection we discuss local graph classes mentioned in the second example after Definition~\ref{defi:local}. Here we have a fixed finite simple graph $L$ and consider the family $\LL$ of all graphs containing $L$ as an induced subgraph. Recall that in this case we let $M_{L,{\rm ind}}(n)$ denote $M_{\LL}(n)$ and similarly, we denote $R_{\LL}(n)$ and $DC(\LL)$ by $R_{L,{\rm ind}}(n)$ and $DC(L,{\rm ind})$, respectively. In this section we
prove that requiring the subgraphs to be induced does not change
the answer from that of subsection ~\ref{subsect:noninduced}. The upper bound, of course, trivially carries over from the non-induced case, while the lower bound strengthens the one in Theorem \ref{thm:DCL}.

\medskip
\par\noindent
\begin{thm}
\label{thm:DCLindb}
For any fixed graph $L$ we have $$DC(L,{\rm ind})=\frac{1}{\chi(L)-1}.$$
\end{thm}
\begin{remark}\label{partition_number}
{\rm Note that despite the apparent similarity,
the proof of Theorem \ref{thm:DCL} does not carry over to show
Theorem \ref{thm:DCLindb}. Still, it is possible to
describe the asymptotic number of induced $L$-free graphs for a fixed graph $L$, by introducing the partition number $r(L)$. Define $r(L)$ as the largest integer $r$ so that there is some integer $s$, $0 \leq s \leq r$ such that the vertices of $L$ cannot be covered by $s$ cliques and $r-s$ independent sets. Results obtained independently by Alekseev \cite{Ale} and by Bollob\'as and Thomason \cite{BT1,BT2} imply that the
number of induced-$L$-free graphs on $n$ vertices
is $2^{(1-1/r(L))n^2/2 +o(n^2)}$. Note that $r(L) \geq \chi(L)-1$, as the vertices of $L$ cannot be covered by $\chi(L)-1$ independent sets. There are cases when equality holds and thus the required result follows (an example for that is $L=C_5$, the $5$-length cycle), but in general, $r(L)$ can be much larger than $\chi(L)-1$, as shown, for example, by any long (even or odd) cycle. Hence the proof of Theorem \ref{thm:DCL} does not give the required bound
for an arbitrary $L$.}
\end{remark}

The main tool in the proof of Theorem \ref{thm:DCLindb} is
Lemma \ref{lemma:kpartiteub}, which bounds the number of balanced $k$-partite induced-$L$-free graphs.

\begin{lemma}\label{lemma:kpartiteub}
For a fixed positive integer $k$ and a fixed graph $L$, consider the set of balanced $k$-partite graphs $G$ on $k$-classes $A_1, \dots, A_k$, each having size $n/k$. Among those graphs,
at most $2^{\big(1-\frac{1}{\chi(L)-1}\big)n^2/2+o(n^2)}$ do not
contain an induced subgraph isomorphic to $L$.
\end{lemma}

In what follows we will assume some familiarity with Szemer\'edi's regularity lemma and the terminology related to it. A good introduction to these notions (with a full proof of the lemma itself) can be found, for example, in Section 7.4 of Diestel's book \cite{Diestel}.

Before presenting the proof of Lemma \ref{lemma:kpartiteub}, whose
proof relies on the regularity lemma, we need to establish the following
auxiliary claim about finding induced subgraphs using regularity. Versions
of this result have been used before, for completeness we
include a simple proof.

\begin{lemma}\label{lemma:regularity}
For every $0<\delta<\frac{1}{2}$ and a graph $L$, there exist positive constants $\eps=\eps(\delta, L)$ and $n_0=n_0(\delta, L)$ with the following property.
Suppose $G$ is a graph whose vertices are partitioned into $\chi(L)$
independent sets, $V_1, V_2, \dots, V_{\chi(L)}$ of equal size which is
at least $n_0$. If the pairs $(V_i, V_j)$ are $\eps$-regular and if their density satisfies $d(V_i, V_j)\in (\delta, 1-\delta)$ for all $i, j$, then $G$ contains an induced copy of $L$.
\end{lemma}
\begin{proof}
Let $l$ be the number of vertices of $L$, and let
$V(L)=U_1\cup \dots \cup U_{\chi(L)}$ be a partition of $L$ into $\chi(L)$ independent sets. The idea is to find a copy of $L$ in $G$ such that the vertices of $U_i$ come from $V_i$, for all $i$.
Having this goal in mind, refine the partition $V_1, \dots, V_{\chi(L)}$
by partitioning
every $V_i$ into $|U_i|$ smaller sets, $V_{i, 1}, \dots, V_{i, |U_i|}$,
of nearly equal sizes. Subdividing $V_i$ does not affect regularity of the new pairs significantly, and so the pairs $(V_{i, j}, V_{k, m})$ are still $l\eps$-regular.
Similarly, the density of the pairs $(V_{i, j}, V_{k, m})$ is in the interval $(\delta-\eps, 1-\delta+\eps)$ if $i\neq k$, and zero otherwise.
Now, we use the following standard lemma (see e.g. Lemma 3.2 in \cite{AFKS}).

\medskip
\par\noindent
{\bf Lemma on regularity and induced subgraphs.} (\cite{AFKS})
{\it For every $0<\delta_0<1$ and $l\in \mathbb{Z}_{>0}$, there
exist positive constants $\eps_0=\eps_0(\delta_0, l)$ and $\mu_0=\mu_0(\delta_0, l)$ with the following property.
Suppose $L$ is a graph on $l$ vertices, $v_1, \dots, v_l$ and
$S_1, \dots, S_l$ is an $l$-tuple of disjoint vertex sets of a large graph $G$ such that every pair $S_iS_j$ is $\eps_0$-regular
with density at least $\delta_0$ if $v_iv_j$ is an edge of $L$ and at most $1-\delta_0$ if $v_iv_j$ is not an edge of $L$. Then, $G$ contains at least $\mu_0 \prod_{i=1}^l |S_i|$ tuples $(w_1, \dots, w_l)\in S_1\times\dots \times S_l$ spanning an induced copy of $L$, where each $w_i$ corresponds to $v_i$.}
\medskip
\par\noindent
Note that identifying the sets $V_{i, j}$ with $S_i$ from this lemma satisfies the conditions, as the pairs $(V_{i, j}, V_{i, k})$ of density $0$ occur only when the corresponding vertices of $L$ belong to the same independent set $U_i$. Therefore, if we set $\delta_0=\delta-\eps, \eps_0=\eps l$ and choose $\eps$ small enough, we conclude that for large
enough $n_0$, $G$ contains at least one induced copy of $L$, as needed.
\end{proof}

\begin{proof} [Proof of Lemma \ref{lemma:kpartiteub}.]
The proof uses Szemer\'edi's regularity lemma to partition an arbitrary induced-$L$-free graph,
after which a standard "cleaning" argument shows that the
main contribution to the total number of such graphs comes from the number of possible bipartite graphs induced by regular pairs of density bounded away from $0$ and $1$. Then, Tur\'an's theorem provides an upper bound for the number of such pairs, completing the proof. The details follow.

Begin by fixing a small $\delta>0$ and the corresponding $\eps=\eps(\delta, L)>0$ from Lemma \ref{lemma:regularity}.
What we have in mind is a not yet fixed induced-$L$-free graph $G$ whose vertex set will be partitioned and we will calculate the number of ways we can connect pairs of vertices so that our graph becomes indeed induced-$L$-free.
By Szemer\'edi's regularity lemma, there is an $\eps$-regular partition $V(G)=V_0\cup V_1\cup \dots\cup V_T$,
where $T$ is a constant, that is, it can be bounded from above
by a number $T_0(\eps,k)$
which does not depend on $n$. We may also assume that this partition refines the original partition $V(G)=A_1\cup\dots\cup A_k$,
i.e., that every $V_i$ for $i\geq 1$ is a subset of some $A_j$.
(This follows from the standard proof of the regularity lemma, cf.
e.g. \cite{Diestel}, that works by iterating the refinement of some
original partition which we can choose to be the one given by
the partite classes $A_i$.)

We split the pairs $(V_i, V_j)$ into three classes - the irregular pairs, the $\eps$-regular pairs with $d(V_i, V_j)\in [0, \delta]\cup [1-\delta, 1]$, and the $\eps$-regular pairs with $d(V_i, V_j)\in (\delta, 1-\delta)$. We first show that the contribution of all pairs except those in the third class to the total number of graphs is negligible, and then we discuss the number of pairs in the third class.

There are at most $(T+1)^n=2^{o(n^2)}$ ways to distribute the vertices
into parts $V_0, \dots, V_{T}$, and there are
at most $2^{\eps n^2}$ ways to choose the edges incident to $V_0$, as $|V_0|\leq \eps n$. Note that there are no edges within the other parts $V_i$, so we have no additional choice here.
The number of ways to associate pairs to classes is at most
$3^{\binom{T}{2}}$, which is a constant (depending on $\eps$).

Further, there are at most $2^{\big(\frac{n}{T}\big)^2}$ ways to choose the edges between each of the irregular pairs, and there are at
most $\eps T^2$ irregular pairs. Hence, the number of choices for the edges between all irregular pairs is at most $2^{\eps n^2}$.
Similarly, for the parts of density close to 0 or 1, we have at
most $2\sum_{i=0}^{\delta (n/T)^2}\binom{(n/T)^2}{i}\leq \frac{\delta n^2}{T^2}\binom{(n/T)^2}{\delta (n/T)^2}\leq \frac{\delta n^2}{T^2}(e\delta^{-1})^{\delta (n/T)^2}$, and there are at most $\binom{T}{2}$ such pairs.
Hence, the total number of choices for the edges in these pairs is
at most $e^{\delta (1-\log \delta) \binom{n}{2}+o(n^2)}$.

Finally, we need to bound the contribution from the pairs
$(V_i, V_j)$ with $d(V_i, V_j)\in (\delta, 1-\delta)$.
Define an auxiliary graph $G'$ on the vertex set $\{V_1, \dots, V_T\}$, in which $V_iV_j$ is an edge if and only if $(V_i, V_j)$
forms an $\eps$-regular pair of density in $(\delta, 1-\delta)$. Lemma \ref{lemma:regularity} shows that
if $G$ contains no induced copy of $L$ then $G'$ cannot contain a subgraph isomorphic to $K_{\chi(L)}$.
Applying Tur\'an's theorem, we conclude that at most
$(1-\frac{1}{\chi(L)-1})T^2/2$ pairs among $V_1, \dots, V_T$
can be $\eps$-regular and have density bounded away from
$0$ and $1$.
Hence, there are at most
$2^{(1-\frac{1}{\chi(L)-1})(T^2/2)\frac{n^2}{T^2}}
\leq 2^{(1-\frac{1}{\chi(L)-1})n^2/2+o(n^2)}$ choices for the edges in this case.

To complete the argument, we let $\delta, \eps\to 0$ and note that the contribution of all pairs except those in the third class is negligible.
We conclude that the number of $k$-partite balanced induced-$L$-free graphs on $n$ vertices is at most
$2^{(1-\frac{1}{\chi(L)-1})n^2/2+o(n^2)}$, as stated. \end{proof}

Using Lemma \ref{lemma:kpartiteub}, the proof of Theorem \ref{thm:DCLindb} follows almost immediately.
\begin{proof}[Proof of Theorem \ref{thm:DCLindb}]
Analogously to what was done in the proof of Theorem \ref{thm:DCL}, now consider the graph
$G_{L,k}$ whose vertices are all balanced $k$-partite graphs
on $n$ vertices, in which two vertices are adjacent if and only if their symmetric difference contains no induced copy of $L$. As shown by Lemma \ref{lemma:kpartiteub}, the maximum degree in this graph is at most
$\Delta(G_{L,k})\leq 2^{\big(1-\frac{1}{\chi(L)-1}\big)n^2/2+o(n^2)}$. On the other hand, $G_{L,k}$ has
$2^{\binom{k}{2}\big(\frac{n}{k}\big)^2}
=2^{(1-\frac{1}{k}) n^2/2}
$ vertices. Hence, the size of its maximum independent set is at least
$$\alpha(G_{L,k})\geq\frac{|V(G_{L,k})|}{\Delta(G_{L,k})+1}\geq
2^{\big(\frac{1}{\chi(L)-1}-\frac{1}{k}\big)n^2/2+o(n^2)}.$$
As $k$ can be arbitrarily large, we conclude that
$DC(L,$ ind$)=\frac{1}{\chi(L)-1}$, completing the proof.
\end{proof}

In the case of bipartite graphs, one can prove a result analogous to Lemma \ref{lemma:kpartiteub} with much better bounds, as shown in \cite{ABBM}, and \cite{AMY} with an improved error term.
\begin{lemma}
\label{l91}
For every fixed bipartite graph $L$ there is some
$\eps=\eps(L)>0$ so that the number of bipartite graphs on two
classes of vertices $A$ and $B$, both of size $m$, that do not
contain an induced copy of $L$ is smaller than
$2^{m^{2-\eps}}$.
\end{lemma}
\medskip

\noindent
The statement of this lemma, with a non-optimal value of $\eps$,
follows from the results in the above mentioned papers that estimate the
number of bipartite graphs that do not contain an induced copy
of the universal bipartite graph $U(k)$ with $k$ vertices in one
vertex class
and $2^k$ in the other, connected in all possible ways to the vertices of
the first class. Every bipartite graph $L$ is an induced copy
of $U(k)$ for all sufficiently large $k$.
Although this is good enough for our purpose here,
we present next a shorter
new proof of the lemma, which gives a tight estimate up to a logarithmic
factor in the exponent for many graphs $L$.
\begin{lemma}
\label{l92}
Let $L$ be a bipartite graph $L$ with color classes of sizes
$s$ and $t \geq s$. Then the
number of bipartite graphs on two
classes of vertices $A$ and $B$, each of size $m$, that do not
contain an induced copy of $L$, is at most
$2^{c(s,t) m^{2-1/s} \log m }$.
\end{lemma}
The above is tight, up to the logarithmic term in the exponent,
for every pair $t \geq s$ where $t$ is sufficiently large as a function
of $s$. Indeed, as shown by the (projective) norm-graphs of
\cite{KRS}, \cite{ARS}, there is a bipartite graph with classes of
vertices of size $m$ each and with $\Omega(m^{2-1/s})$ edges that
contains no copy of the complete bipartite graph $K=K_{s,t}$, provided
$t>(s-1)!$ Every member of the
collection of all subgraphs of this graph contains
no (induced  or non-induced) copy of $K$.

\begin{proof}
We apply the Sauer-Perles-Shelah Lemma (\cite{Sa}, \cite{Sh}) which states
that for any collection $\C$ of more than $\sum_{i=0}^d {q \choose i}$
functions from a set $Q$ of size $q$ to $\{0,1\}$ there is a subset
$D \subset Q$ of cardinality $d+1$ {\em shattered} by $\C$. That is,
for any function
$g:D \mapsto \{0,1\}$ there is an $f \in \C$ such that
$g(x)=f(x)$ for all $x \in D$. Let $L$  be a bipartite graph with classes
of vertices of sizes $t \geq s$ and let $\G$ be a collection of
graphs on the two vertex classes $A$ and $B$, where $|A|=|B|=m$.
Let $d=ex(2m,K_{s,t})$ be the maximum possible
number of edges in a
graph on $2m$ vertices that
contains no copy of the complete bipartite graph $K_{s,t}$.
By the K\H{o}v\'ari-S\'os-Tur\'an Theorem \cite{KST},
$d \leq b(s,t)m^{2-1/s}$.
By the Sauer-Perles-Shelah Lemma, if
$$
|\G| \geq 1+\sum_{i=0}^d {{m^2} \choose i}
$$
then the collection of graphs $\G$, viewed as a collection of
functions from the set of all edges of the complete
bipartite graph on the classes of vertices $A,B$
to $\{0,1\}$, shatters a set of $d+1$ edges. By the definition
of $d$ this set of edges
contains a complete bipartite graph with vertex classes of sizes
$s$ and $t$. Any subgraph of this graph (and in particular $L$)
is an induced subgraph of some member of $\G$. Since the right-hand-side
of the last inequality is smaller than
$$m^{2d} \leq m^{2b(s,t) m^{2-1/s}}=2^{c(s,t)m^{2-1/s} \log m},
$$
this completes the proof.
\end{proof}

\medskip
\par\noindent
The following is a counterpart of Corollary~\ref{DCgen} for the induced case.

\begin{cor}\label{DCindgen}
Let $\G$ be a set of graphs, each containing at least one edge,
and let $\LL_{(\G,{\rm ind})}$ be the local graph class containing all graphs that contain at least one $G\in\G$ as an induced subgraph.
Then
$$DC(\LL_{(\G,{\rm ind})})= \frac{1}{\chi_{\rm min}(\G)-1}.$$
\end{cor}

\proofka
The upper bound is immediate from $DC(\LL_{({\G},{\rm ind})})\le DC(\LL_{\G})$ and Corollary~\ref{DCgen}. The lower bound follows, as in Corollary~\ref{DCgen}, from the fact that $DC(\LL_{({\G},{\rm ind})})
\geq DC(G, {\rm ind})$ for any $G\in \G$.
In particular, picking a graph $G$ with $\chi(G)=\chi_{\rm min}(\G)$ suffices.
\qed

\begin{remark}\label{rem:indual}
{\rm It is straightforward that the above results imply a strengthening (as far as the upper bound is concerned) of the statement in Remark~\ref{rem:asydual}, namely that we have for the dual problem also in the induced case
$$\lim_{n\to\infty}\frac{2}{n(n-1)}\log D_{\LL_{({\G},{\rm ind})}}(n)=1-\frac{1}{\chi_{\min}(\G)-1}.$$
This follows from Lemma~\ref{lem:ub} and Corollary~\ref{DCindgen}
just as the statement in Remark~\ref{rem:asydual} followed from Lemma~\ref{lem:ub} and Corollary~\ref{DCgen}.
$\Diamond$}
\end{remark}

\begin{remark}\label{rem:empty}
{\rm If $L$ is the
edgeless graph on $r \geq 2$ vertices then it is clear that
for every $n \geq r$ there is  a family  of
$2^{{n \choose 2}-{r \choose 2}}$ graphs on $[n]$ so that the symmetric
difference between any two contains an induced copy of $L$. Indeed
we simply take all graphs which agree on the ${r \choose 2}$ edges
of a fixed $r$-clique. Note that for every fixed $r$ this is a constant
fraction of all graphs on $n$ vertices, much larger than
$M_L(n)$ or $M_{L,{\rm ind}}(n)$ for any graph $L$ with at least
$2$ edges.
This (including the constant) is clearly tight for $r=2$ (the family
cannot contain a graph and its complement), and by the main result of
\cite{EFF} it is tight also for $r=3$ (the result in \cite{EFF} holds
for any family in which any two members agree on a triangle, not only
if any two intersect in a common triangle-the
equivalence of these  two statements
is proved already in \cite{CFGS}).  We do not know if this is tight
for larger values of $r$. Note that an equivalent formulation of the
question here is the determination of $M_{\F}(n)$ for $\F$ which
is the family of all graphs with independence number at least $r$.
$\Diamond$}
\end{remark}

\section{Open problems}

In this final section we collect some
related problems left open.
\begin{problem}
For what graph families $\F$ is it true that $M_{\F}(n)$ is achieved
by a linear graph family code, that is one that is closed under
the symmetric difference operation?
\end{problem}

Our results here include examples where this is the case as well as
ones in which it is not. Indeed in Theorem \ref{thm:fullstar} the
precise answer is $n$ or $n+1$, and if this is not a power of $2$
there is no optimal linear solution. Another family of examples in which
the optimal family cannot be achieved by a linear example is that in which
the family $\F$ is the family of all graphs with at most $2r$ edges,
where $r$ is chosen so that the sum
$$
\sum_{i=0}^r {{n \choose 2} \choose i}
$$ is not a power of $2$. Indeed, by a theorem of Kleitman
\cite{Kl} (for usual codes)
the size of the optimum family here is the size of the family
of all graphs  with at most $r$ edges.

\medskip 
The construction in the proof of Theorem~\ref{thm:conn} has the property that for any two of its graphs $G$ and $G'$ with an equal number of edges (that is trivially necessary for satisfying the condition in the next problem) their two {\em asymmetric} differences $$G\setminus G'=([n],E(G)\setminus E(G'))\ {\rm and}\ G'\setminus G=([n],E(G')\setminus E(G))$$ are isomorphic. This suggests the following question.

\begin{problem}
What is the maximum possible size of a graph family ${\cal A}$ of graphs on $n$ vertices satisfying that if $A,A'\in {\cal A}$ then $A\setminus A'$ and $A'\setminus A$ are isomorphic?
\end{problem}

\medskip
\par\noindent
A larger family
than the one we can obtain from the construction in the proof of Theorem~\ref{thm:conn} can be given by taking $\lfloor n/k\rfloor$ vertex-disjoint stars, each on $k$ vertices. This gives a lower bound that is superexponential in $n$ but we do not have any nontrivial upper bound.

\medskip
Theorems~\ref{thm:Hp} and \ref{thm:fullstar} show a huge difference between requiring a spanning path or a spanning star in the symmetric differences. One may wonder what happens ``in between''. Note that if we formulate this ``in betweenness'' so that we want to have a spanning tree with diameter at most $k$, then while with $k=2$ we are at Theorem~\ref{thm:fullstar} and with $k=n-1$ at Theorem~\ref{thm:Hp}, already for $k=3$ we get the same result as for $k=n-1$ by the construction in the proof of Theorem~\ref{thm:conn}. (This is simply because complete bipartite graphs contain spanning trees of diameter at most $3$.) So it seems plausible to formulate questions in terms of more specific ``natural'' sequences of spanning trees $T_1, T_2, \dots$. (In the problem below the notation $M_{T_n}(n)$ is meant to denote the largest possible cardinality of a family of graphs on vertex set $[n]$ such that the symmetric difference of any two of them contains $T_n$ as a subgraph.)

\medskip
\par\noindent
\begin{problem}
For what ``natural'' sequences $T_1, T_2,\dots, T_i,\dots$ of trees (with $T_i$ having exactly $i$ vertices for every $i$) will the value of $M_{T_n}(n)$ grow only linearly in $n$? A similar question is valid if $T_i$ is
replaced by
${\cal T}_i$, some ``natural'' family of $i$-vertex trees.
\end{problem}

\medskip
\par\noindent
Propositions~\ref{prop:triv34}, \ref{lem:MK3_5}, \ref{lem:MK3_6}, \ref{lem:MC7}
showed that the upper bound of Proposition~\ref{thm:Wilex} can be sharp for small values of $n$ for the requirement that a triangle or at least an odd cycle is contained in the symmetric differences. It would be interesting to know whether this can also happen for large values of $n$.

\medskip
\par\noindent
\begin{problem}
Is $$M_{K_3}(n)= 2^{{n\choose 2}-\lceil\frac{n}{2}\rceil\lfloor\frac{n}{2}\rfloor}$$ true always or at least for infinitely many values of $n$? Even if this is not so, does the analogous equality hold for $M_{{\cal C}_{\rm odd}}(n)$?
\end{problem}
Note that there are much better known estimates for the number of triangle-free
graphs on $n$ labeled vertices than the one we have used here, in fact, it
is known that almost all of these graphs are bipartite \cite{EKRo}. While this
improves the gap between the upper and lower bounds that follow from
our proofs for
$M_{K_3}(n)$, it is still far from determining its precise value.

\medskip
\par\noindent
The final  problem we mention is related to the remark
in the end of the last subsection.

\medskip
\par\noindent
\begin{problem}
Is it true that for any fixed $r>3$ the maximum possible
cardinality of a family of graphs on $n$ labeled vertices in which
the symmetric difference between any two members has independence
number at least $r$,
is exactly a $1/{r \choose 2}$ fraction of
the number of all graphs on these vertices ?
\end{problem}

\end{document}